\documentclass [11pt,reqno]{amsart}
\usepackage{amsmath,amssymb,verbatim,geometry,color}
\usepackage[all]{xy}

\usepackage{mathrsfs}
\usepackage[backref,pagebackref,pdftex,hyperindex]{hyperref}
\usepackage[pdftex]{graphicx}

\usepackage{tikz}
\usepackage{tikz-cd}

\def\XXint#1#2#3{{\setbox0=\hbox{$#1{#2#3}{\int}$ }
\vcenter{\hbox{$#2#3$ }}\kern-.6\wd0}}

\newcommand{\B}{\mathbb{B}}
\newcommand{\C}{\mathbb{C}}
\newcommand{\G}{\mathbb{G}}

\newcommand{\N}{\mathbb{N}}
\renewcommand{\P}{\mathbb{P}}
 
 \newcommand{\R}{\mathbb{R}}
 \newcommand{\Z}{\mathbb{Z}}

\newcommand{\cA}{\mathcal{A}}

\newcommand{\cD}{\mathcal{D}}
\newcommand{\cE}{\mathcal{E}}

\newcommand{\cX}{\mathcal{X}}

\newcommand{\om}{\omega}

\DeclareMathOperator{\jj}{J}

\DeclareMathOperator{\DF}{DF}

\DeclareMathOperator{\Exc}{Exc}

\DeclareMathOperator{\Amp}{Amp}

\DeclareMathOperator{\Vol}{Vol}

\DeclareMathOperator{\ord}{ord}

\DeclareMathOperator{\PSH}{PSH}

\DeclareMathOperator{\tr}{tr}

\DeclareMathOperator{\Ric}{Ric}

\newcommand{\ddc}{dd^c}
\newcommand{\dc}{d^c}

\newcommand{\NA}{\mathrm{NA}}

\numberwithin{equation}{section}       % Number formulas within sections
\newtheorem{thm}{Theorem}[section]
\newtheorem{prop}[thm] {Proposition}
\newtheorem{defi}[thm] {Definition}
\newtheorem{lem}[thm] {Lemma}

\newtheorem{prop-def}[thm]{Proposition-Definition}

\newtheorem{rmk}[thm]{Remark}

\newtheorem{mainthm}{Theorem}

\makeatletter
\newtheorem*{rep@theo}{\rep@title}
\newcommand{\newreptheo}[2]{%
\newenvironment{rep#1}[1]{%
 \def\rep@title{#2 \ref{##1}}%
 \begin{rep@theo}}%
 {\end{rep@theo}}}
\makeatother

\newreptheo{theo}{Theorem}

\newreptheo{conjecture}{Conjecture}

\newreptheo{corol}{Corollary}

\newreptheo{Defin}{Definition}

 \theoremstyle{plain}
\newtheorem*{namedthm}{\namedthmname}
\newcounter{namedthm}

\makeatletter

\theoremstyle{remark}

 \usepackage{hyperref}
\hypersetup{
    unicode=false,        
    pdftoolbar=true,      
    pdfmenubar=true,       
    pdffitwindow=false,     
    pdfstartview={FitH},    
    pdftitle={On the transcendental YTD Conjecture},    
    pdfauthor={Trusiani},     
    colorlinks=true,       
   linkcolor=blue,          
    citecolor=blue,        
    filecolor=black,      
    urlcolor=blue}

\title[On the transcendental YTD Conjecture]{On the transcendental Yau-Tian-Donaldson Conjecture}
\date{} % Activate to display a given date or no date (if empty),
\author{Antonio Trusiani}

\address{Università di Roma Tor Vergata\\
Via della Ricerca Scientifica 1, 00133\\
Roma, Italy}
\email{trusiani@mat.uniroma2.it}

\begin{document}

\maketitle
 
\begin{abstract}
    We prove the (uniform) Yau-Tian-Donaldson Conjecture for transcendental K\"ahler classes in the case of trivial automorphisms. Indeed, we show the existence of \emph{Special} K\"ahler Fujita Approximations of big cohomology classes associated to big test configurations, establishing the equivalence between the uniform $K$-stability and the strengthened version for models. More generally, we show that any big cohomology class on a compact K\"ahler manifold admits such Special K\"ahler Fujita Approximations: the volume and the analogue of the Riemann-Roch coefficient of the big class are both approximated by those of K\"ahler classes on higher compactifications.
\end{abstract}

\section{Introduction}
The main goal of this manuscript is to establish the following solution to the transcendental Yau-Tian-Donaldson Conjecture (\cite{Yau87, Tian97, Don02}).
\begin{mainthm}[{Solution to the Yau-Tian-Donaldson Conjecture, case with trivial automorphisms}]\label{thmA}
    Let $\alpha$ be a K\"ahler class on a compact K\"ahler manifold $X$.
    Then the following are equivalent:
    \begin{itemize}
        \item[i)] there exists a unique cscK metric in $\alpha$;
        \item[ii)] $(X,\alpha)$ is uniformly $K$-stable. 
    \end{itemize}
\end{mainthm}
It is well-known that both $(i), (ii)$ require that the identity component of the \emph{linear automorphism group} $\mathrm{LAut}(X)$ \cite{Fuj78} is trivial. Thus, Theorem \ref{thmA} establishes a link between Differential and Algebraic Geometry in these cases, partially extending \cite[Thm. A]{Tru26}. We refer to the companion paper \cite{Tru26} and references therein for the algebraic setting of Theorem \ref{thmA}, i.e. when the cohomology class $\alpha$ is integral; here, we exclusively consider the transcendental case.

Extending the integral case, (uniform) $K$-stability for general K\"ahler classes was first introduced and investigated in \cite{DR17, Sjo18}, where the authors independently showed the implication $(i)\Rightarrow (ii)$ of Theorem \ref{thmA}. The reader may also be interested in the subsequent papers \cite{Der18, Sjo20} where a (geodesic, equivariant) notion of $K$-polystability has been proved to be necessary for the existence of a cscK metric in the case of general linear automorphism groups.

Note that a uniform version of $K$-stability is expected to be necessary for characterizing the existence of a unique cscK metric, and hence the correspondence in Theorem \ref{thmA} is suspected to be optimal (cf. \cite{ACGTF08, Hat21, Liu26}, chronologically).

\subsection*{Special K\"ahler Fujita Approximations} As better explained below, following \cite{Li21, Tru26, MP25, MPWN25}, a proof of Theorem \ref{thmA} can be deduced from the following refinement of the renowned Fujita Approximation Theorem \cite{Fuj94, Bou02}.
\begin{mainthm}[{On the existence of Special K\"ahler Fujita Approximations}]\label{thmB}
    Let $\alpha$ be a big cohomology class on an $n$-dimensional compact K\"ahler manifold. Then there are modifications\footnote{The morphisms $p_k$ can be chosen to be given by a sequence of blowups along smooth centers.} $p_k:Y_k\to X$, K\"ahler classes $\beta_k$ and effective $\R$-divisors $E_k$ such that
    \begin{itemize}
        \item[i)] $p_k^*\alpha=\beta_k+\{E_k\}$;
        \item[ii)] $\beta_k^n\longrightarrow \Vol(\alpha)$ as $k\to +\infty$;
        \item[iii)] $\beta_k^{n-1}\cdot K_{Y_k}\longrightarrow \langle \alpha^{n-1}\rangle\cdot K_X$ as $k\to +\infty$. 
    \end{itemize}
\end{mainthm}
By \cite{Bou02}, there always exist data $(p_k:Y_k\to X, \beta_k,E_k)$ satisfying $(i)$ and $(ii)$, yielding K\"ahler Fujita Approximations of $\alpha$ and extending \cite{Fuj94}. Thus, the additional convergence $(iii)$ of the first Riemann–Roch coefficients is the new requirement; any K\"ahler Fujita Approximation of $\alpha$ that also satisfies $(iii)$ will be said to be \emph{Special}. A version of Theorem \ref{thmB} for integral big classes on smooth projective varieties was already proved in the companion paper \cite{Tru26}.

\subsection*{On the proofs of Theorems \ref{thmA}, \ref{thmB}} As mentioned before, the implication $(i)\Rightarrow (ii)$ in Theorem \ref{thmA} was independently proven by Sj\"ostrom Dyrefelt and Dervan-Ross in \cite{DR17, Sjo18}, so we discuss the converse $(ii)\Rightarrow (i)$.
Partly generalizing to the transcendental setting the Non-Archimedean formalism mainly developed by Boucksom-Jonsson and collaborators (see for instance \cite{BHJ17,BFJ15,BFJ16,BJ22,BJ25I,BJ23II, Li22b, Li21}), in \cite{MP25} Piccione showed that the existence of a unique cscK metric can be deduced from a strengthened version of uniform $K$-stability, i.e., the so-called uniform $K$-stability over $\cE^{1,\NA}$. Furthermore, in \cite{MPWN25}, the latter notion has been proved to be equivalent to the uniform $K$-stability \emph{for models} (see \cite{Li22b} for the algebraic version of such a result). Roughly speaking, while (uniform) $K$-stability requires the positivity of a certain weight, the \emph{Donaldson-Futaki invariant}, associated with \emph{K\"ahler test configurations} $(\cX,\cA+\cD)$, the $K$-stability for models asks that such positivity holds for \emph{big} test configurations, i.e. when the cohomology classes $\cA+\cD$ are big instead of K\"ahler (see section \ref{sec:2} for more details). Thus, similarly to \cite{Tru26}, Theorem \ref{thmB} solves the transcendental analogue of Boucksom-Jonsson's Regularization Conjecture on the Non-Archimedean Entropy \cite{BJ18}. In particular, the Donaldson-Futaki invariant of big test configurations can be continuously approximated by those of K\"ahler test configurations, and the uniform $K$-stability implies the uniform $K$-stability for models; hence, the findings in the aforementioned paper \cite{MPWN25} conclude the proof of Theorem \ref{thmA}. 
\medskip

Although the proof of Theorem \ref{thmB} follows the same strategy explored in \cite{Tru26}, there are some key differences due to the lack of algebraic tools/results. In \cite{Tru26} a canonical Fujita Approximation was constructed by blowing-up the base ideals $\mathfrak{b}_m$, and the Castelnuovo-Mumford regularity Theorem plus Nadel's vanishing Theorem were used to control the difference $\ord_F\mathfrak{b}_m-\ord_F\mathcal{J}\left(\mathfrak{b}_m\right)$ for any $F$ prime divisor over $X$. Here $\mathcal{J}\left(\mathfrak{b}_m\right)$ is the multiplier ideal sheaf of $\mathfrak{b}_m$. Finally, the key discrepancy bound on normalized blowups established in \cite[Prop. 3.1]{Tru26} was exploited to deduce that $\beta_k^{n-1}\cdot K_{Y_k/X}\to 0$ as $k\to +\infty$. In this manuscript, we start by observing that the analytic proof of such a key discrepancy bound adapts to compact K\"ahler manifolds (see Proposition \ref{thm:Discrepancies_Bouck}). Moreover, in Lemma \ref{lem:Special_H}, we prove that the convergence $\beta_k^{n-1}\cdot K_{Y_k/X}\to 0$ is equivalent to condition $(iii)$ of Theorem \ref{thmB}.
Thus, after a standard perturbation argument, the problem reduces to constructing a \emph{Nef} Fujita Approximation (i.e. the $\beta_k$ are nef classes) by blowing-up a sequence of ideals $\mathcal{I}_k$ for which we can control $\ord_F\mathcal{I}_k-\ord_F \mathcal{J}(\mathcal{I}_k)$. However, we cannot rely on the properties of the base ideals as in the algebraic setting. 

Thus, following \cite{DNTT24} and the author's previous works \cite{Tru20b, Tru20c}, we first observe that Nef Fujita Approximations can be encoded in suitable classes of singularities of $\theta$-plurisubharmonic functions, where $\theta\in \alpha$ is a fixed closed $(1,1)$-form (see Lemma \ref{lem:ClassicFujita}). This allows us to pass from the data $(Y_k,\beta_k)$ to elements $\varphi_k\in \PSH(X,\theta)$, gaining the possibility to work on a fixed variety with pluripotential-theoretical techniques. Using this perspective, we construct a Nef Fujita Approximation whose singularities are encoded in the ideal sheaves $\mathcal{I}_k:=\mathcal{J}(kV_\theta)\cdot \mathcal{O}_X(-k_0E)$. Here $k_0\in\N$ is fixed, $V_\theta$ is the largest non-positive $\theta$-plurisubharmonic function and $E$ is an effective divisor such that the \emph{Non-K\"ahler locus of $\alpha$} satisfies $\mathrm{E}_{\mathrm{nK}}(\alpha)=\mathrm{Supp}(E)$ (we can reduce to this setting by passing to a higher model). Next, the key part consists in proving that
\begin{equation}
    \label{eqn:Key}
    \mathcal{J}(\mathcal{I}_k)\cdot \mathcal{O}_X(-k_0E)\subset \mathcal{I}_k
\end{equation}
for any $k\in\N$. The calculations performed in \cite{Pop03} are essential in showing \eqref{eqn:Key} (see Lemma \ref{lem:Uniformity}); hence the aforementioned perspective, which interprets Nef Fujita Approximations in terms of classes of singularities in $\PSH(X,\theta)$, is particularly advantageous in the proof of Theorem \ref{thmB}.
\medskip

Let me also stress the recent solution to the \emph{Orthogonality Conjecture} (hence the differentiability of the volume function, the transcendental holomorphic Morse inequalities, and the duality between the pseudoeffective and the movable cone) established in \cite{Tos26} is not necessary to prove Theorem \ref{thmA}. Indeed, as explained in Remark \ref{rmk:Final}, for the proof of Theorem \ref{thmB} we only need to know that any $\gamma$ big class satisfies
$$
\langle \gamma^{n-1}\rangle\cdot F=0
$$
for any prime divisor $F$ contained in the non-K\"ahler locus of $\gamma$. Such a result for big test configurations can be deduced from \cite{Vu23,Nys24, MPWN25}.

\subsection*{Relation with Darvas-Zhang paper}
Recently in \cite{DZ26}, Darvas and Zhang proved a Yau-Tian-Donaldson correspondence between the existence of a unique cscK metric and a \emph{quantized} version of uniform $K$-stability, the uniform $K^\beta$\emph{-stability}. We refer to \cite{DZ26} for its definition, and to \cite{Pap26,LZ26,DZ26b} for further studies concerning this notion. Here we collect the following result.
\begin{mainthm}\label{thmC}
    Let $\alpha$ be a K\"ahler class on a compact K\"ahler manifold $X$. Then the following are equivalent:
    \begin{itemize}
        \item[i)] $(X,\alpha)$ is uniformly $K$-stable;
        \item[ii)] $(X,\alpha)$ is uniformly $K^\beta$-stable for some $\beta>0$.
    \end{itemize}
\end{mainthm}
Indeed, thanks to Theorem \ref{thmA} and \cite[Thm 5.5]{DZ26}, both stability notions of Theorem \ref{thmC} are equivalent to the existence of a unique cscK metric in $\alpha$.

\subsection*{Future developments} 
A natural next step is to establish Theorem \ref{thmA} in the \textquotedblleft presence of automorphisms\textquotedblright.
%The most natural subsequent work of the present manuscript concerns proving a version of Theorem \ref{thmA} in the \textquotedblleft presence of automorphisms\textquotedblright.
The Special K\"ahler Fujita Approximation of Theorem \ref{thmB} is constructed in a sufficiently canonical way to ensure the $\G$-invariance for a given algebraic group $\G$ acting on $(X,\alpha)$. Thus, similarly to \cite{Tru26}, it is plausible to believe that Theorem \ref{thmB} will provide a proof of the equivalence between an equivariant version of $K$-stability and its version for models. Therefore, the essential missing point here is to extend the findings of \cite{MP25, MPWN25} to the case with automorphisms.
\smallskip

Similarly, analogous results concerning the implication from \emph{weighted} versions of uniform $K$-stability for models to the existence of special metrics can be explored in the \emph{weighted} and singular setting in the future (see \cite{PTT23, BJT26, Sze25, PT26, Ino22, Lah19, Lah23, AJL23, DJL24, DJL25, HL25a, HL25b, BJ26} and references therein). Then, an adaptation of Theorem \ref{thmB} may again represent the bridge between the uniform weighted $K$-stability and the version for models, thus establishing the historically harder implication in the Yau-Tian-Donaldson correspondence.
\smallskip

Finally, a second version of \cite{DNTT24} will soon appear, where Theorem \ref{thmB} is exploited to obtain the convexity of the Mabuchi functional in big cohomology classes without extra hypotheses.

\subsection*{Structure of the paper}

In Section \ref{sec:2} all necessary preliminaries are introduced. Special K\"ahler Fujita Approximations are introduced in Section \ref{sec:4}, while Section \ref{sec:5} is dedicated to the proof of Theorem \ref{thmB}. Finally in Section \ref{sec:6} we explain in more detail how Theorem \ref{thmB} leads to Theorem \ref{thmA}.

\subsection*{Disclosure on the use of AI} The author used ChatGPT (GPT-$5.6$) for language editing, in particular to check grammar, style, typos, and clarity of exposition. All mathematical arguments and statements are the sole responsibility of the author.

\section*{Acknowledgments}
The author thanks S. Boucksom for his mathematical comments and suggestions. The author is also grateful to S. Trapani for carefully reading drafts of the paper, to P. Piccione for fruitful discussions and to T.  Papazachariou for useful comments.
    
\section{Preliminaries}\label{sec:2}
Let $(X,\om)$ be a compact K\"ahler manifold of (complex) dimension $n$, where $\om$ is a fixed K\"ahler form. Set $\dc:= \frac{i}{4\pi}(\bar{\partial}-\partial)$ so that $\ddc=\frac{i}{2\pi}\partial \bar{\partial}$.

Recall that a function $u:X\to \R\cup\{-\infty\}$ is said to be \emph{quasi-plurisubharmonic} (q-psh) if locally $u\overset{loc}{=}g+\varphi$ for $g\in C^{\infty}$ and $\varphi$ psh. The set of q-psh functions is equipped with the $L^1$-topology and with a natural partial order: $u$ is less singular than $v$ ($u\succcurlyeq v$) if $u+C\geq v$ for a constant $C$.

Given a smooth closed $(1,1)$-form $\theta$, a q-psh function $u$ such that $\theta_u:=\theta+\ddc u\geq 0$ as $(1,1)$-current is said to be $\theta$-psh and $\PSH(X,\theta)$ denotes the set of all these functions. Note that any closed and positive current $T$ representing the cohomology class $\alpha:=\{\theta\}\in \mathrm{H}^{1,1}(X,\R)$ can be written as $T=\theta_u$ for $u\in\PSH(X,\theta)$ and the choice of $u$ is unique modulo translation by constants. The class $\alpha=\{\theta\}$ is then said to be \emph{big} if $\PSH(X,\theta-\varepsilon\om)$ is not empty for any $\varepsilon>0$ small enough. Namely there are \emph{K\"ahler currents} in $\alpha$, i.e. closed and positive currents $T\in\alpha$ such that $T\geq \varepsilon \om$.

\subsection{Some properties of the non-pluripolar product}\label{ssec:Non_Pluripolar}
Following the pioneering works of Bedford-Taylor \cite{BT82, BT87}, in \cite{BEGZ10} the authors defined the \emph{non-pluripolar product}
$
\langle T_1\dots T_p\rangle
$
among closed and positive $(1,1)$-currents $T_1,\dots, T_p$. We refer to \cite{BEGZ10} and references therein for the definition and for a deep analysis of such a product. Here we just collect some of its properties.
\begin{prop}[{\cite[Prop. 1.4, Rem. 1.7, Thm. 1.8]{BEGZ10}}]\label{prop:NPP}
    Let $T_1,\dots,T_r$ be closed and positive currents. Then
    \begin{itemize}
        \item[i)] $(T_1,\dots,T_r)\longrightarrow\langle T_1\wedge \cdots \wedge T_r \rangle$ is symmetric, multilinear and does not put mass on pluripolar sets;
        \item[ii)] $\langle T_1\wedge \cdots \wedge T_r \rangle$ is a closed and positive $(r,r)$-current;
        \item[iii)] if $T_r$ is smooth then $\langle T_1\wedge \cdots \wedge T_r \rangle=\langle T_1\wedge \cdots \wedge T_{r-1}\rangle\wedge T_r$ as closed and positive $(r,r)$-currents;
        \item[iv)] if $p:Y\to X$ is a bimeromorphic holomorphic map (i.e. $p$ is a \emph{modification}) among K\"ahler manifolds then $\langle T_1\wedge \cdots \wedge T_r \rangle=p_*\langle p^*T_1\wedge \cdots \wedge p^*T_r \rangle$.
    \end{itemize}
\end{prop}
Moreover, within the same cohomology classes, the (total mass of the) non-pluripolar product respects the partial order of the q-psh functions. Namely,
\begin{equation}\label{eqn:Monotonicity}
    \int_X\langle T_1\wedge T_2\wedge \cdots \wedge T_n\rangle\leq \int_X\langle T_1'\wedge T_2\wedge \cdots \wedge T_n\rangle
\end{equation}
if $T_1,T_1',T_2,\dots, T_n$ are closed and positive $(1,1)$-currents and $T_1=\theta_u, T_1'=\theta_{u'}$ where $u\preccurlyeq u'$ (see \cite[Thm. 1.2]{WN17}, \cite[Thm. 1.1]{DDNL17b}, \cite[Thm. 1.16]{BEGZ10}).

\subsection{Volume as Non-Pluripolar Monge-Ampère mass.}\label{ssec:Volume}
The volume of a class $\alpha=\{\theta\}$ such that $\PSH(X,\theta)\neq \emptyset$ (i.e. $\alpha$ is \emph{pseudoeffective}) is defined as
$$
\Vol(\alpha):=\sup_{T\in\alpha}\int_X \langle T^n \rangle=\int_X \langle \theta_{V_\theta}^n \rangle 
$$
where $V_\theta:=\sup\{u\in\PSH(X,\theta)\, : \, u\leq 0\}$ is the least singular non-positive $\theta$-psh function and where the last equality follows from the monotonicity of the non-pluripolar product. We refer to \cite[Def. 1.3]{Bou02} for the original definition of the volume of a big (more generally, psef) class and to \cite[Prop. 1.18]{BEGZ10} for a proof of the equivalence with the definition given in terms of the non-pluripolar product. We also recall that the volume produces a continuous function on the psef cone such that $\Vol(\alpha)>0$ if and only if $\alpha$ is big, and that if $\alpha$ is nef then $\Vol(\alpha)=\alpha^n$ \cite[Prop. 4.3]{Bou02}.

\subsection{Functions with analytic singularities}
Letting $\mathcal{I}\subset \mathcal{O}_X$ be a coherent ideal sheaf and $c>0$, a q-psh function $u$ is said to have \emph{analytic singularities of type $(\mathcal{I},c)$} if locally
$$
u\overset{\mathrm{loc}}{=} g+ c\log\Big(\sum_j\lvert f_j\rvert^2\Big)
$$
where $g$ is bounded while $(f_j)_j$ are local generators of the ideal $\mathcal{I}$. Observe that if $u$ has analytic singularities of type $(\mathcal{I},c)$ then it has analytic singularities of type $(\overline{\mathcal{I}},c)$ where $\overline{\mathcal{I}}$ is the integral closure of $\mathcal{I}$. Denote by $\mathcal{A}(X,\theta)$ the set of $\theta$-psh functions with analytic singularities. 
Any big class contains many functions with analytic singularities thanks to the renowned Regularization Theorem of Demailly \cite{Dem92} (cf. \cite[Thm. 3.2]{DP04}).

\subsection{Non-K\"ahler locus}\label{ssec:Non-Kahler}
Let $\alpha=\{\theta\}\in \mathrm{H}^{1,1}(X,\R)$ be a cohomology class. The \emph{non-K\"ahler locus} of $\alpha$ is defined as
$$
\mathrm{E}_{\mathrm{nK}}(\alpha):=\bigcap_{\mathcal{I}}V(\mathcal{I})
$$
where $\mathcal{I}$ varies among all the coherent ideal sheaves for which there exists a $(\theta-\varepsilon\om)$-psh function with analytic singularities of type $\left(\mathcal{I},c\right)$ for some $\varepsilon,c>0$, i.e. it is the intersection of all the singularity loci of all K\"ahler currents in $\alpha$ with analytic singularities.
Its complementary $\Amp(\alpha):=X\setminus \mathrm{E}_{\mathrm{nK}}(\alpha)$ is called \emph{Ample locus} of $\alpha$, and $\Amp(\alpha)\neq \emptyset$ if and only if $\alpha$ is big. The non-K\"ahler locus first appeared in \cite[Def. 3.16]{Bou04}. Using Demailly Regularization Theorem \cite{Dem92} and the strong Noetherian property, \cite[Thm. 3.17.(ii)]{Bou04} shows that if $\alpha$ is big then there exists $u\in\mathcal{A}(X,\theta-\varepsilon\om)$ with analytic singularities of type $\left(\mathcal{I},c\right)$ such that $\mathrm{E}_{\mathrm{nK}}(\alpha)=V(\mathcal{I})$. Moreover if $\alpha$ is big and nef then
\begin{equation*}%\label{eqn:Null_Locus}
    \mathrm{E}_{\mathrm{nK}}\left(\alpha\right)=\bigcup_{\alpha^{\dim V}\cdot V=0 } V
\end{equation*}
where the union is over all irreducible analytic subsets $V\subset X$ by \cite[Thm. 1.1]{CT15}.

\subsection{Orthogonality Conjecture}
Let $\alpha=\{\theta\}\in \mathrm{H}^{1,1}(X,\R)$ be a big class. For any $r=1,\dots,n$ we denote by $\langle \alpha^r \rangle$ the cohomology class of the closed and positive $(r,r)$-current $\langle \theta_{V_\theta}^r\rangle$. The \emph{positive intersection product} $\langle \alpha^r\rangle$ has been introduced analytically in \cite{BDPP13} (following the PhD thesis \cite{BouThesis}). The definition given coincides with the original one thanks to the monotonicity \eqref{eqn:Monotonicity} and the discussion performed in \cite[Sec. 1.5]{BEGZ10}.

The following recent solution to the \emph{Orthogonality Conjecture} is a key ingredient in Theorem \ref{thmB}.
\begin{thm}[{\cite[Cor. 1.2, Cor. 1.3]{Tos26}}]\label{conj:Ort}
    Let $\alpha$ be a big class on a compact K\"ahler manifold $X$. Then
    $$
    \Vol(\alpha)= \langle \alpha^{n-1} \rangle\cdot \alpha.
    $$
    In particular a prime divisor $D$ belongs to the non-K\"ahler locus $\mathrm{E}_{\mathrm{nK}}(\alpha)$ if and only if
    \begin{equation}
        \label{eqn:Orthogonality}
        \langle\alpha^{n-1} \rangle\cdot D=0.
    \end{equation}
\end{thm}
Theorem \ref{conj:Ort} was previously proved on projective manifolds: for integral classes, see \cite[Sec. 4]{BDPP13}, \cite[Thm. 4.15, Thm. B]{BFJ09}, \cite[Thm. C]{ELMNP09}; for general transcendental classes, see \cite[Thm. D]{WN19}. %Although there have been recent developments (see for instance \cite{Nys24,Vu23}), for general K\"ahler manifolds the \emph{Orthogonality Conjecture} is still open and it is connected to the duality of the psef and the movable cone, to the differentiability of the volume function on the big algebraic cone and to Demailly's holomorphic Morse inequalities. 

\subsection{Lelong numbers and Siu Decomposition}\label{ssec:LelongSiu}
We recall that the \emph{Lelong number} of a q-psh function $u$ at a point $x$ is defined as
\begin{equation}
    \label{eqn:Lelong}
    \nu(u,x):=\sup\left\{c\geq 0\, : \, u(z)\leq c\log \lVert z-x\rVert^2+O(1) \text{ near } x\right\}.
\end{equation}
The Lelong number of $u$ along a subvariety $V$ is given by
$
\nu(u,V)=\inf_{x\in V} \nu(u,x).
$

Given $T=\theta_u$ closed and positive current, $\nu(T,x):=\nu(u,x)$ is well-defined since \eqref{eqn:Lelong} is local and $\theta$ is smooth. By a deep result of Siu \cite{Siu74}, any closed and positive $(1,1)$-current $T$ admits a unique decomposition
\begin{equation}\label{eqn:Siu}
	T=R+\sum_F \nu(T,F)[F]
\end{equation}
where the sum is over all prime divisors $F$, $[F]$ is the current of integration along $F$, $R$ is a closed and positive current such that $\nu(R,F)=0$ for any $F$, $\nu(T,F_j)>0$ for at most a countable family of prime divisors $F_j$ and the series $\sum_j\nu(T,F_j)[F_j]$ weakly converges in the sense of currents; \eqref{eqn:Siu} is known as \emph{Siu Decomposition} of the current $T$ (cf. \cite[Ch. III, (8.16)]{DemNotes}).

Note that if $u$ is a q-psh function with analytic singularities of type $\left(\mathcal{I},c\right)$ then $\nu(u,F)=c\ord_F \mathcal{I}$ for any prime divisor $F$ on $X$ where the order of $\mathcal{I}$ along $F$ is defined as the infimum of the vanishing orders of $f$ varying $f\in\mathcal{I}$. In particular if $p:Y\to X$ is a log resolution for $\mathcal{I}$ then
\begin{equation}\label{eqn:Lelong_Anal_Sing}
    \nu(u\circ p,F)=c\ord_F D
\end{equation}
for any prime divisor $F$ on $Y$ where $D$ is the effective $\Z$-divisor such that $\mathcal{O}_Y\cdot p^{-1}\mathcal{I}=\mathcal{O}_Y(-D)$. Thus if $T=\theta_u$ for $u\in\mathcal{A}(X,\theta)$ of type $\left(\mathcal{I},c\right)$ and $p:Y\to X$ is a log resolution of $\mathcal{I}$ then the Siu Decomposition of $p^*\theta_u$ is given by
$$
p^*\theta_u=S+c[D]
$$
for $S$ closed and positive $(1,1)$-current with \emph{bounded potentials} (i.e. locally $S\overset{loc}{=}\ddc v$ for $v$ bounded).

\subsection{Multiplier ideal sheaves}\label{ssec:Multiplier} Given a coherent ideal sheaf $\mathcal{I}\subset \mathcal{O}_X$, its associated \emph{multiplier ideal sheaf} $\mathcal{J}(\mathcal{I})$ is defined as
$$
\mathcal{J}(\mathcal{I})=p_*\mathcal{O}_Y\left(K_{Y/X}-E\right)
$$
where $p:Y\to X$ is any fixed log resolution of the ideal $\mathcal{I}$ and where $E$ is the effective divisor such that $\mathcal{O}_Y\cdot p^{-1}\mathcal{I}=\mathcal{O}_Y(-E)$. If $\varphi$ is a q-psh function with analytic singularities of type $\left(\mathcal{I},1\right)$ then $\mathcal{J}(\mathcal{I})=\mathcal{J}(\varphi)$ where $\mathcal{J}(\varphi)$ is the analytic sheaf of germs of local holomorphic functions $f$ such that $\lvert f\rvert^2e^{-\varphi}$ is locally integrable with respect to the Lebesgue measure (see \cite[Thm. 9.3.42]{Laz04II}).
This analytic interpretation shows that multiplier ideal sheaves are integrally closed.
We also recall that for any modification $p:Y\to X$ and any prime divisor $F\subset Y$ we have $\ord_F p^{-1}\mathcal{J}(\varphi)\cdot \mathcal{O}_Y\leq \nu(\varphi\circ p,F)$ (see \cite[Lem. B.4]{BBJ15}). In particular if $\phi$ is a q-psh function with analytic singularities of type $(\mathcal{J}(\varphi),1)$ then $\varphi\preccurlyeq \phi$. Indeed, letting $\om$ be a K\"ahler form such that $\varphi,\phi\in \PSH(X,\om)$, this follows by comparing the potentials of the closed and positive currents $p^*\om_{\phi}-[D]$, $p^*\om_{\varphi}-[D]$ where $p:Y\to X$ is a log resolution of $\mathcal{J}(\varphi)$ such that $p^{-1}\mathcal{J}(\varphi)\cdot \mathcal{O}_Y=\mathcal{O}_Y(-D)$.

\subsection{CscK metrics}\label{ssec:CscK}
A \emph{constant scalar curvature K\"ahler metric} (cscK metric) in a cohomology class $\alpha$ is a K\"ahler metric associated to a K\"ahler form $\om\in \alpha$ such that the scalar curvature $S(\om)=\tr_{\om}\Ric(\om)=n\frac{\Ric(\om)\wedge\om^{n-1}}{\om^n}$ is constant, i.e. it is equal to $\overline{S}=\frac{\int_X S(\om)\om^{n}}{\int_X \om^{n}}=n\frac{c_1(X)\cdot\alpha^{n-1}}{\alpha^n}$.
The typical examples of cscK metrics are the so-called \emph{K\"ahler-Einstein metrics} in which the K\"ahler forms are proportional to their Ricci forms.

   \subsection{(Uniform) $K$-stability for transcendental classes}\label{ssec:K-stability} We are going to briefly describe the notion of (uniform) $K$-stability for transcendental K\"ahler classes (see \cite{Sjo18}, \cite{DR17}).

    Let $\alpha$ be a K\"ahler class on an $n$-dimensional compact K\"ahler manifold $X$. Using the notation in \cite[Def. 1.1.1]{MPWN25}, a \emph{test configuration}\footnote{test configurations in the transcendental setting were first considered in \cite[Def. 2.10]{DR17}, \cite[Def. 1.3]{Sjo18}} $\left(\cX,\mathcal{A}+\mathcal{D}\right)$ for $(X,\alpha)$ consists of
    \begin{itemize}
        \item[i)] a compact K\"ahler manifold\footnote{Similarly to \cite{MPWN25} it is enough to consider \emph{smooth dominating} test configurations when studying $K$-stability.} $\cX$ together with a surjective morphism $\pi:\cX\to X\times \P^1$ such that $\cX\setminus \cX_0 \overset{\pi}{\simeq}X\times\left(\P^1\setminus\{0\}\right)$ where $\cX_0:=(p_2\circ\pi)^{-1}(0)$ denotes the fibre over $0\in\P^1$;
        \item[ii)]  a lift of the standard $\C^*$-action on $X\times \P^1$ to $\cX$ making $\pi$ equivariant;
        \item[iii)] a class $\mathcal{A}+\mathcal{D}$ on $\cX$ where $\mathcal{A}:=\pi^*p_1^*\alpha$ while $\mathcal{D}$ is (the class of) an effective $\R$-divisor supported on $\cX_0$.
    \end{itemize}
    A test configuration is said to be K\"ahler (resp. big) if the cohomology class $\mathcal{A}+\mathcal{D}$ is K\"ahler (resp. big). Moreover, it is said to be \emph{trivial} if $\cX=X\times \P^1$ with the trivial action; in such a case, $\cD$ is necessarily a multiple of the zero fiber. Following \cite{Der16, BHJ17}, the triviality is measured through the functional
    \begin{equation}
        \label{eqn:JInv}
        \jj^\NA(\cX,\cA+\cD):=\cA^n\cdot (\cA+\cD)- \frac{(\cA+\cD)^{n+1}}{n+1},
    \end{equation}
    which is non-negative for any K\"ahler test configurations and vanishes if and only if $(\cX,\cA+\cD)$ is trivial (\cite[Thm. 1.5]{Der18}, \cite[Lem. 5.3]{Sjo20}).
    Then, similarly to the algebraic setting, $K$-stability asks for the positivity of a certain weight attached to K\"ahler test configurations: the so-called \emph{Donaldson-Futaki invariant}
    $$
    \DF(\cX,\cA+\cD):=\overline{S}\frac{(\cA+\cD)^{n+1}}{n+1}+(\cA+\cD)^n\cdot K_{\cX/\P^1},
    $$
    where $\overline{S}$ is the average of the scalar curvature (cf. subsection \ref{ssec:CscK}).
    \begin{defi}
        Let $(X,\alpha)$ be an $n$-dimensional compact K\"ahler manifold. Then $(X,\alpha)$ is said to be
        \begin{itemize}
            \item[i)] \emph{$K$-stable} if
            $$
            \DF(\cX,\cA+\cD)\geq 0
            $$
            for any K\"ahler test configuration $(\cX,\cA+\cD)$ with equality if and only if $(\cX,\cA+\cD)$ is trivial;
            \item[ii)] \emph{uniformly $K$-stable} if there exists $\sigma>0$ such that
            $$
            \DF(\cX,\cA+\cD)\geq \sigma \jj^\NA(\cX,\cA+\cD)
            $$
            for any K\"ahler test configuration $(\cX,\cA+\cD)$.
        \end{itemize}
    \end{defi}
    Clearly, uniform $K$-stability implies $K$-stability, while, as stated in the Introduction, there is evidence that the reverse implication may not be valid in general. Finally, again as mentioned in the Introduction, we recall that if $(X,\alpha)$ is (uniformly) $K$-stable, then the identity component of the \emph{linear automorphism group} $\mathrm{LAut}(X)$ must be trivial (\cite[Cor. 4.17]{Sjo20}).

   \subsection{$K$-stability for models}
   In \cite{MP25, MPWN25} the notion of K-stability has been strengthened to ensure the existence of cscK metrics. Generalising \cite{Li22b, Li21} to the transcendental setting, the Donaldson-Futaki invariant and the Non-Archimedean $J$-functional \eqref{eqn:JInv} can be extended to big test configurations $(\cX,\cA+\cD)$ as
   \begin{gather*}
       \DF(\cX,\cA+\cD)=\overline{S}\frac{\langle (\cA+\cD)^{n+1}\rangle}{n+1}+ \langle (\cA+\cD)^n\rangle\cdot K_{\cX/\P^1},\\
       \jj^\NA(\cX,\cA+\cD)=\cA^n\cdot \langle \cA+\cD\rangle-\frac{\langle(\cA+\cD)^{n+1} \rangle}{n+1}.
   \end{gather*}
   Then, $(X,\alpha)$ is said to be \emph{uniformly $K$-stable for models} if there exists $\sigma>0$ such that
   $$
   \DF(\cX,\cA+\cD)\geq \sigma \jj^\NA(\cX,\cA+\cD)
   $$
   for any big test configuration $(\cX,\cA+\cD)$. Since by \cite[Cor. 5.5]{Sjo18}, \cite[Thm. 1.1]{DR17} the existence of a unique cscK metric implies uniform $K$-stability, the following recent result gives a partial Yau-Tian-Donaldson correspondence.
   \begin{thm}[{\cite[Thm. A]{MPWN25}}]\label{thm:Pietro_David}
       If $(X,\alpha)$ is uniformly $K$-stable for models then there exists a unique cscK metric in $\alpha$.
   \end{thm}

\section{K\"ahler Fujita Approximations}\label{sec:4}
In the whole section $\alpha$ will be a big cohomology class on a compact K\"ahler manifold $(X,\omega)$, and $\theta\in\alpha$ will be a fixed smooth closed $(1,1)$-form.

\smallskip

We introduce the following definition of (Nef) Kähler Fujita Approximations.

\begin{defi}
    The data $(p_k:Y_k\to X, \beta_k,E_k)$ is said to be a \emph{K\"ahler Fujita Approximation of} $\alpha$ if $p_k:Y_k\to X$ are modifications from compact K\"ahler manifolds $Y_k$ such that
    \begin{itemize}
        \item[i)] $p_k^*\alpha=\beta_k+\{E_k\}$ for $E_k$ effective $\R$-divisors and $\beta_k$ K\"ahler classes;
        \item[ii)] $ \beta_k^n\longrightarrow \Vol(\alpha)$ as $k\to +\infty$.
    \end{itemize}
    If the $\beta_k$ are only nef then we say that $(p_k:Y_k\to X, \beta_k, E_k)$ is a \emph{Nef} Fujita Approximation of $\alpha$.
\end{defi}
In the Nef setting the top intersection product $\beta_k^n$ still captures the volume of the classes and there is no loss of generality in assuming that the morphisms $p_k$ are given by a composition of blow-ups along smooth centers.

Following the pioneering work of Fujita \cite{Fuj94} in the algebraic setting, it can be proved that K\"ahler Fujita Approximations always exist (see \cite[Thm. 1.4]{Bou02}).

\subsection{Pluripotential-theoretical perspective of Fujita Approximations.}

Let $p:Y\to X$ be a modification and let $D$ be an effective $\R$-divisor such that $\beta:=p^*\alpha-\{D\}$ is a psef cohomology class. Then, for any closed and positive current $S\in\beta$ there exists a unique closed and positive current $T\in \alpha$ such that $p^*T=S+[D]$ (see, for instance, \cite[Prop. 1.2.7]{BouThesis}). Following \cite[Def. 4.12]{DNTT24}, we can interpret Nef Fujita Approximations from a pluripotential-theoretic perspective as follows.

\begin{lem}\label{lem:ClassicFujita}
    Let $(p_k:Y_k\to X, \beta_k,E_k)$ be a (resp. Nef)  K\"ahler Fujita Approximation of $\alpha$. Then there exists $\{\varphi_k\}_k\subset \PSH(X,\theta)$ such that
    \begin{itemize}
        \item[i)] $p_k^*\theta_{\varphi_k}=\eta_k+[E_k]$ for $\eta_k\in \beta_k$ K\"ahler form (resp. $\eta_k$ closed and positive current with minimal singularities representing a nef class $\beta_k$);
        \item[ii)] $V_{\varphi_k}:=\int_X \langle \theta_{\varphi_k}^n \rangle=\beta_k^n$ for any $k\in\N$ and $V_{\varphi_k}\longrightarrow \int_X \langle \theta_{V_\theta}^n \rangle= \Vol(\alpha)$ as $k\to +\infty$. 
    \end{itemize}
    Conversely, if $\{\varphi_k\}_k\subset\PSH(X,\theta)$ is a sequence such that $(i)$ and $(ii)$ hold, then $(p_k:Y_k\to X, \beta_k,E_k)$ is a (resp. Nef) K\"ahler Fujita Approximation of $\alpha$.
\end{lem}
The choice of the reference form $\theta$ is clearly not significant. Given a different form $ \theta'\in \alpha$ one can replace $\varphi_k\in\PSH(X,\theta)$ with $\varphi_k':=\varphi_k+f\in \PSH(X,\theta')$ for $f\in C^\infty(X)$ such that $\theta=\theta'+\ddc f$.
\begin{proof}
    Let us first prove the case of K\"ahler Fujita Approximation. Letting $\eta_k\in\beta_k$ be a K\"ahler form, $\eta_k+[E_k]$ is a closed and positive current in $p^*_k \alpha$. Thus, by the preceding discussion, there exists a unique $S=\theta_{\varphi_k}$ such that $p_k^*\theta_{\varphi_k}=\eta_k+[E_k]$ and $(i)$ follows. Then Proposition \ref{prop:NPP} gives
    $$
    \beta_k^n=\int_{Y_k} \eta_k^n=\int_{Y_k}\langle p_k^*\theta_{\varphi_k}^n\rangle=\int_X \langle\theta_{\varphi_k}^n \rangle=V_{\varphi_k},
    $$
    concluding the case of K\"ahler Fujita Approximations. In the Nef case the proof proceeds in the same way since $\beta_k^n=\int_{Y_k}\langle\eta_k^n\rangle$ if $\eta_k$ is a closed and positive current with minimal singularities by the discussion in subsection \ref{ssec:Volume}. 
\end{proof}

\subsection{Special K\"ahler/Nef Fujita Approximations}
By abuse of terminology, given a big class $\alpha$, we define its first Riemann-Roch coefficient as
$$
\tau_1(X,\alpha)=\langle \alpha^{n-1}\rangle \cdot  K_X,
$$
extending \cite[Definition 4.1]{Li21}. 
We can now introduce the notion of \emph{Special} Fujita Approximations, following \cite{Li21, Tru26}.
\begin{defi}
    A K\"ahler/Nef Fujita Approximation $(p_k:Y_k\to X, \beta_k,E_k)$ of $\alpha$ is said to be \emph{Special} if
    $$
    \tau_1\left(Y_k,\beta_k\right)\longrightarrow \tau_1(X,\alpha)
    $$
    as $k\to +\infty$.
\end{defi}
If $\beta$ is a nef class on a compact K\"ahler manifold $Y$ then $\tau_1(Y, \beta)=\beta^{n-1}\cdot K_Y$, and by the Projection Formula $\tau_1(Z,q^*\beta)= \tau_1(Y,\beta)$ if $q:Z\to Y$ is a birational morphism from another compact K\"ahler manifold $Z$. Thus, without loss of generality, the morphisms $p_k:Y_k\to X$ of any Special Nef Fujita Approximation $(p_k:Y_k\to X, \beta_k,E_k)$ of $\alpha$ can be assumed to be given as compositions of blowups along smooth centers.

\smallskip

The following result has been proved in the algebraic setting in \cite[Lem. 4.3]{Tru26} (following \cite[Lemma 4.10]{Li21}). We provide here a direct proof that does not use the transcendental holomorphic Morse inequalities (and the equivalent differentiability of the volume) recently established in \cite{Tos26}.
\begin{lem}\label{lem:Special_H}
    A Nef Fujita Approximation $(p_k:Y_k\to X, \beta_k,E_k)$ of $\alpha$ is Special if and only if \begin{equation}
        \label{eqn:H_phi}
        \beta_k^{n-1}\cdot K_{Y_k/X}\longrightarrow 0    \end{equation}
    as $k\to +\infty$.
\end{lem}
\begin{proof}
    As $\beta_k$ are nef, $\beta_k^{n-1}=\langle \beta_k^{n-1}\rangle$. Thus
    $$
    \beta_k^{n-1}\cdot K_{Y_k/X}= \tau_1(Y_k,\beta_k)-\beta_k^{n-1}\cdot p_k^*K_X,
    $$
    and to conclude the proof it suffices to show that
    $
    \beta_k^{n-1}\cdot p_k^*K_X\rightarrow \tau_1(X,\alpha)
    $
    as $k\to +\infty$. Letting $\sigma$ be a smooth closed $(1,1)$-form representing $K_X$ and letting $\eta_k$ be closed and positive $(1,1)$-currents with minimal singularities representing $\beta_k$, from \cite[Lem. 4.10]{DNTT24} we have
    $$
    \beta_k^{n-1}\cdot p_k^*K_X=\int_{Y_k}\langle \eta_k^{n-1}\wedge p_k^*\sigma\rangle.
    $$
    Thus Proposition \ref{prop:NPP} and Lemma \ref{lem:ClassicFujita} give
    $$
    \beta_k^{n-1}\cdot p_k^*K_X=\int_{Y_k}\langle \eta_k^{n-1}\wedge p_k^*\sigma\rangle=\int_{Y_k}\langle p_k^*\theta_{\varphi_k}^{n-1}\wedge p_k^*\sigma \rangle=\int_X \langle \theta_{\varphi_k}^{n-1}\wedge \sigma \rangle.
    $$
    where $\varphi_k\in\PSH(X,\theta)$ is such that $p_k^*\theta_{\varphi_k}=\eta_k+[E_k]$.    
    Finally, applying Lemma \ref{lem:ConvFuj} below, we obtain
    $$
    \int_X\langle \theta_{\varphi_k}^{n-1}\wedge \sigma\rangle \longrightarrow \int_X\langle \theta_{V_\theta}^{n-1}\wedge \sigma \rangle=\langle \alpha^{n-1} \rangle\cdot K_X=\tau_1(X,\alpha),
    $$
    concluding the proof.
\end{proof}

    \begin{lem}\label{lem:ConvFuj}
    Let $\{\varphi_k\}_k\subset \PSH(X,\theta)$ be a sequence associated to a Nef Fujita Approximation as in Lemma \ref{lem:ClassicFujita}. Then
    \begin{equation}
        \label{eqn:To_Prove}
        \int_X \langle \theta_{\varphi_k}^p\wedge T_1\wedge \cdots \wedge T_{n-p}\rangle \longrightarrow \int_X \langle \theta_{V_\theta}^p\wedge T_1\wedge \cdots \wedge T_{n-p} \rangle
    \end{equation}
    as $k\to +\infty$ for any $p=0,\dots,n$ and for any closed and positive currents $T_1,\dots,T_{n-p}$.
\end{lem}
\begin{proof}
    This result is essentially well-known to experts; for completeness, we provide a proof.
    To simplify the notation, $\langle T\wedge S \rangle:=\langle T\wedge T_1\wedge \cdots \wedge T_{n-p}\rangle$ for any closed and positive $(p,p)$-current $T$.
    Thanks to \cite[Thm. A, Thm. B]{BEGZ10} and \cite[Thm. 1.4]{DDNL17b} we can pick $u$ as the unique solution to the complex Monge-Ampère equation
    $$
    \begin{cases}
        \langle \theta_u^n\rangle=\frac{\Vol(\alpha)}{\{\om\}^n} \,\om^n\,\\
        u\in \PSH(X,\theta), \sup_X u=0\\
        u-V_\theta\in L^\infty
    \end{cases},
    $$
    and let $u_k$ be the unique solution to the complex Monge-Ampère equation \emph{with prescribed singularities}
    \begin{equation}
        \label{eqn:MA_Prescr}
        \begin{cases}
            \langle \theta_{u_k}^n\rangle=\frac{V_{\varphi_k}}{\{\om\}^n} \,\om^n\,\\
            u_k\in \PSH(X,\theta), \sup_X u_k=0\\
            u_k-\varphi_k\in L^\infty
        \end{cases}.
    \end{equation}
    Note that the existence of the solution to \eqref{eqn:MA_Prescr} is given by \cite[Thm. A]{DDNL18b} since $\varphi_k$ has \emph{relatively} minimal singularities (see for instance \cite[Rem. 4.4]{DNTT24}).
    By the monotonicity of the non-pluripolar product \eqref{eqn:Monotonicity} we have
    $$
    \int_X \langle \theta_{u_k}^p\wedge S \rangle=\int_X \langle \theta_{\varphi_k}^p\wedge S \rangle\leq \int_X \langle \theta_{V_\theta}^p\wedge S \rangle=\int_X \langle \theta_u^p\wedge S \rangle.
    $$
    Thus, combining \cite[Thm. 1.2]{DDNL17b} with \cite[Thm. 1.4]{DDNL19}, the desired convergence \eqref{eqn:To_Prove} would follow if the \emph{singularity types} $[u_k]$ $d_S$\emph{-converged} to that of $[V_\theta]$ in the sense of \cite{DDNL19}. In our setting, this $d_S$-convergence is equivalent to
    \begin{equation}\label{eqn:Last_Conv}
        \int_X \langle \theta_{\varphi_k}^j\wedge \theta_{V_\theta}^{n-j} \rangle\longrightarrow \Vol(\alpha)
    \end{equation}
    for any $j=0,\dots,n$ by \cite[Lem. 3.6]{DDNL19}. Finally, \eqref{eqn:Last_Conv} is a consequence of \cite[Thm. B]{DDNL18b} since the total mass $V_{\varphi_k}:=\int_X \langle \theta_{\varphi_k}^n\rangle$ converges to $\Vol(\alpha)= \int_X \langle \theta_{V_\theta}^n\rangle$.
\end{proof}

\section{Existence of Special K\"ahler Fujita Approximations}\label{sec:5}
In this section we will prove Theorem \ref{thmB}, noting also that the considered morphisms $p_k:Y_k\to X$ in Theorem \ref{thmB} can be taken to be compositions of blow-ups along smooth centres such that the exceptional loci are simple normal crossing.

Let $\alpha$ be a big class on a compact K\"ahler manifold $X$ and let $\theta\in \alpha$ be a fixed smooth closed $(1,1)$-form. The following control on the discrepancies is the key ingredient.

\begin{prop}[{\cite[Prop. 3.1]{Tru26}}]\label{thm:Discrepancies_Bouck}
    Let $\mathcal{I}\subset \mathcal{O}_X$ be a coherent ideal sheaf, denote by $\mathcal{J}(\mathcal{I})\subset \mathcal{O}_X$ its multiplier ideal sheaf and let $p:Y\to X$ be the normalized blow-up along $\mathcal{I}$. Then
    \begin{equation}
        \label{eqn:Bouck}
        \ord_F K_{Y/X}\leq 2n\left(\ord_F \mathcal{I} - \ord_F \mathcal{J}(\mathcal{I})\right)
    \end{equation}
    for any prime divisor $F\subset Y$.
\end{prop}
\begin{proof}
    Letting $D$ be the effective divisor such that $\mathcal{O}_Y\cdot p^{-1}\mathcal{I}=\mathcal{O}_Y(-D)$, it is well-known that $-D$ is p-ample as a consequence of the construction of the normalized blow-up (see, for instance, \cite[Lem. 8.2]{Bou17Notes}). Thus, the same analytic proof of \cite[Prop. 3.1]{Tru26} applies, where clearly $\om$ is a fixed K\"ahler form not necessarily belonging to an integral class.
\end{proof}

We will also need the following result, which is well known to the experts.
\begin{lem}\label{lem:Functions_vs_Ideals}
    Let $u_1, u_2$ be two q-psh functions with analytic singularities respectively of type $(\mathcal{I}_1,1)$ and $(\mathcal{I}_2,1)$ for integrally closed coherent ideal sheaves $\mathcal{I}_1,\mathcal{I}_2$. Then $u_1\preccurlyeq u_2$ if and only if $\mathcal{I}_1\subset \mathcal{I}_2$.
\end{lem}
\begin{proof}
    If $\mathcal{I}_1\subset \mathcal{I}_2$ then, taking open sets on which $\mathcal{I}_1$ and $\mathcal{I}_2$ admit local generators, we immediately obtain that $u_1\preccurlyeq u_2$ from the definition of functions with analytic singularities and the compactness of $X$.

    Conversely, assume that $u_1\preccurlyeq u_2$ and let $p:Y\to X$ be a common log resolution of $\mathcal{I}_1,\mathcal{I}_2$ and let $D_1,D_2$ be the two effective divisors such that $\mathcal{O}_Y\cdot p^{-1} \mathcal{I}_i =\mathcal{O}_Y(-D_i)$ for $i=1,2$. Let also $\om$ be a K\"ahler class such that $u_1,u_2\in \PSH(X,\om)$. As $\nu(u_1\circ p ,F)\geq \nu(u_2\circ p,F)$ for any prime divisor $F$, from the Siu Decompositions of $p^*\om_{u_1},p^*\om_{u_2}$ (see subsection \ref{ssec:LelongSiu}) we deduce that $D_1\geq D_2$. In particular $\mathcal{I}_1=p_*\mathcal{O}_Y(-D_1)\subset p_*\mathcal{O}_Y(-D_2)=\mathcal{I}_2$.
\end{proof}

    We are now ready to prove Theorem \ref{thmB}. 
    
    \begin{proof}[Proof of Theorem \ref{thmB}] The proof follows the strategy of \cite[Thm. B]{Tru26}.

    For a modification $p:Y\to X$ between compact K\"ahler manifolds, by \cite[Prop. 2.5]{Tos18} we have $\mathrm{E}_{\mathrm{nK}}(p^*\alpha)=p^{-1}\mathrm{E}_{\mathrm{nK}}(\alpha)\cup \Exc(p)$. In particular, by Theorem \ref{conj:Ort}, $\langle (p^*\alpha)^{n-1}\rangle\cdot K_{Y/X}=0$ since $\mathrm{Supp}(K_{Y/X})\subset \Exc(p)$. Therefore $\alpha$ admits a Special K\"ahler Fujita Approximation if and only if $p^*\alpha$ does, and without loss of generality we can and will assume that the non-K\"ahler locus $\mathrm{E}_{\mathrm{nK}}(\alpha)$ has divisorial support. Therefore, letting $\theta\in \alpha$ be a smooth representative and $\om$ be a fixed K\"ahler form, there exists $\psi\in \mathcal{A}(X,\theta)$ such that $\theta_\psi=T+c[E]$ for $T$ closed and positive current with bounded potentials satisfying $T\geq\delta\om$ for $\delta>0$, $E$ effective divisor such that $\mathrm{E}_{\mathrm{nK}}(\alpha)=\mathrm{Supp}(E)$ and $c>0$ (see subsections \ref{ssec:Non-Kahler}, \ref{ssec:LelongSiu}). After rescaling $\alpha$ and $\omega$, we may and will assume that $c=1,\delta=1$.
    
    \textbf{Step 1: It suffices to prove the existence of Special Nef Fujita Approximations.}    
    Suppose that $(p_k:Y_k\to X, \beta_k,E_k)$ is a Special Nef Fujita Approximation of $\alpha$. Without loss of generality, we may assume that $p_k$ is given as compositions of blowups along smooth centers. In particular, there are $p_k$-exceptional effective $\R$-divisors $F_k$ such that $p_k^*\{\om\}-\{F_k\}$ are K\"ahler classes. 
    A simple calculation gives
    \begin{align*}
        p_k^*\alpha &= a(\beta_k+\{E_k\})+ (1-a)\left(p_k^*\{T\} +p_k^*\{E\}\right)\\
        &= a \beta_k + (1-a)(p_k^*\{T\}-\{F_k\}) + a\{E_k\}+ (1-a)p^*_k \{E\} + (1-a)\{F_k\}
    \end{align*}
    for any $a\in (0,1)$. Set $\beta_{k,a}:=a\beta_k+(1-a)(p_k^*\{T\}-\{F_k\})$ and $E_{k,a}:= aE_k+ (1-a)p^*_k E + (1-a) F_k$, noting that  $E_{k,a}$ are effective $\R$-divisors while $\beta_{k,a}$ are K\"ahler classes. Indeed, $T-\om$ is a closed and positive current with bounded potentials; thus, $p_k^*\{T\}-\{F_k\}=p_k^*\{T-\om\}+p_k^*\{\om\}-\{F_k\}$ is K\"ahler being the sum of a nef class and a K\"ahler class.
    By continuity, we can then choose $a_k\in(0,1)$ so that
    \begin{gather*}
        \beta_{k,a_k}^n\geq \beta_k^n-\frac{1}{k},\quad \quad \beta_{k,a_k}^{n-1}\cdot K_{Y_k/X}\leq \beta_k^{n-1}\cdot K_{Y_k/X}+\frac{1}{k}
    \end{gather*}
    for any $k\in\N$. It follows from Lemma \ref{lem:Special_H} that $(p_k:Y_k\to X, \beta_{k,a_k}, E_{k,a_k})$ is a Special K\"ahler Fujita Approximation of $\alpha$.

    \textbf{Step 2: A bound for multiplier ideal sheaves.}
    Set $\mathfrak{a}_k:=\mathcal{J}(kV_\theta)$ and $\mathfrak{b}_k:=\mathcal{J}(\mathfrak{a}_k)$. We want to prove that there exists a constant $k_0\in \N$ such that
    \begin{equation}\label{eqn:Desired_Inclusion}
        \mathfrak{b}_k \cdot \mathcal{O}_X(-k_0E)\subset \mathfrak{a}_k
    \end{equation}
    for any $k\in\N$. By Lemma \ref{lem:Uniformity} below (cf. \cite[Prop. 0.1]{Pop03}) there exists a sequence $\{\phi_k\}_k$ of q-psh functions with analytic singularities of type $\left(\mathfrak{a}_k,1\right)$ and a constant $C>0$ such that $k\theta+\ddc \phi_k\geq -C\om$ for any $k\in\N$. Applying Lemma \ref{lem:Uniformity} again, there also exists a sequence of q-psh functions $\{\tilde{\phi}_k\}_{k}$ with analytic singularities of type $(\mathfrak{b}_k,1)$ such that $k\theta + \ddc \tilde{\phi}_k \geq -2C \om$. As $\psi\in \mathcal{A}(X,\theta-\omega)$, letting $k_0\in\N $ such that $k_0\geq 2C$, we deduce that $\tilde{\phi}_k + k_0\psi\in \mathcal{A}\left(X,(k+k_0)\theta\right)$. In particular
    $$
    \tilde{\phi}_k+k_0\psi\preccurlyeq (k+k_0)V_\theta\preccurlyeq \phi_{k+k_0}
    $$
    for any $k\in\N$ (cf. subsection \ref{ssec:Multiplier}). As multiplier ideal sheaves are integrally closed (see subsection \ref{ssec:Multiplier}), from Lemma \ref{lem:Functions_vs_Ideals} we obtain
    $$
    \mathfrak{b}_k\cdot \mathcal{O}_X(-k_0E)\subset \mathfrak{a}_{k+k_0},
    $$
    and the desired inclusion \eqref{eqn:Desired_Inclusion} follows from $\mathfrak{a}_{k+k_0}\subset\mathfrak{a}_k$.

    \textbf{Step 3: The chosen Special Nef Fujita Approximation.} We define
    $$
    \varphi_k:=\frac{1}{k+k_0}\left(\phi_k+k_0\psi\right)\in\PSH(X,\theta),
    $$
    and we want to prove that $\{\varphi_k\}_k$ is a Special Nef Fujita Approximation. We start by noticing that $\phi_k/k\succcurlyeq V_\theta$ as  $\phi_k/k$ has analytic singularities of type $(\mathcal{J}(kV_\theta),1/k)$ and $\phi_k/k\in \PSH\left(X,\theta+\frac{C}{k}\om\right)$. Thus, as $k_0\geq 2C$ we have
    \begin{align*}
        \Vol(\alpha)&\geq \int_X \langle \theta_{\varphi_k}^n \rangle=\int_X \left\langle\left( \frac{k}{k+k_0}\left(\theta+\frac{C}{k}\om+\ddc \phi_k/k\right)+\frac{k_0}{k+k_0}\left(\theta-\frac{C}{k_0}\om+\ddc \psi\right) \right)^n\right\rangle\\
        &\geq \left(\frac{k}{k+k_0}\right)^n\int_X \left\langle\left(\theta+\frac{C}{k}\omega +\ddc \phi_k/k\right)^n \right\rangle\geq \left(\frac{k}{k+k_0}\right)^n\int_X \left\langle\left(\theta+\frac{C}{k}\omega +\ddc V_\theta\right)^n \right\rangle\\
        &\geq \left(\frac{k}{k+k_0}\right)^n\int_X \langle(\theta_{V_\theta})^n \rangle=\left(\frac{k}{k+k_0}\right)^n\Vol(\alpha)
    \end{align*}
    by the monotonicity of the non-pluripolar product (cf. subsection \ref{ssec:Non_Pluripolar}). In particular $\int_X \langle \theta_{\varphi_k}^n \rangle\longrightarrow \Vol(\alpha)$, i.e. $\{\varphi_k\}_k$ produces a Nef Fujita Approximation by Lemma \ref{lem:ClassicFujita} since $\varphi_k$ has analytic singularities of type $\left( \mathfrak{a}_k\cdot \mathcal{O}_X(-k_0E),1/(k+k_0) \right)$. 
    
    Considering a log resolution $p_k:Y_k\to X$ of $\mathcal{I}_k:=\mathfrak{a}_k\cdot \mathcal{O}_X(-k_0E)$, by the universal property of normalized blowups we have the following commutative diagram
    \[\begin{tikzcd}
    Y_k \arrow{rr}{q_k} \arrow[swap]{dr}{p_k} & & Z_k \arrow{dl}{\pi_k} \\
    & X &
    \end{tikzcd},
    \]
    where $\pi_k:Z_k\to X$ is the normalized blow-up along $\mathcal{I}_k$. Since $J(\mathcal{I}_k)=\mathfrak{b}_k\cdot O_X(-k_0E)$, by \eqref{eqn:Desired_Inclusion} it follows that
    $$
    \ord_F \mathcal{I}_k-\ord_F J(\mathcal{I}_k)\leq \ord_F  \mathcal{O}_X(-k_0E)
    $$
    for any $F$ over $X$, i.e. for any prime divisor $F\subset Y$ on an $n$-dimensional normal compact K\"ahler variety $Y$ that dominates $X$ through a modification $p:Y\to X$. Thus, Proposition \ref{thm:Discrepancies_Bouck} leads to
    \begin{equation}
        \label{eqn:New_Oct}
        K_{Z_k/X}\leq 2nk_0 \pi_k^*E        
    \end{equation}
    for any $k\in\N$, as Weil divisors. 

    On the other hand, letting $D_k,E_k$ be the effective divisor such that $\mathcal{O}_{Z_k}\cdot \pi_k^{-1}\mathcal{I}_k=\mathcal{O}_{Z_k}(-D_k)$ and $\mathcal{O}_{Y_k}\cdot p_k^{-1}\mathcal{I}_k=\mathcal{O}_{Y_k}(-E_k)$, by construction there exist closed and positive currents $T_k,S_k$ with bounded potentials such that
    $$
    T_k+\frac{1}{k+k_0}[E_k]=p^*_k \theta_{\varphi_k}= q_k^*\pi_k^*\theta_{\varphi_k}= q^*_k\left(S_k+\frac{1}{k+k_0}[D_k]\right)=q_k^*S_k + \frac{1}{k+k_0}[q_k^*D_k].
    $$
    By uniqueness of the Siu Decomposition (see subsection \ref{ssec:LelongSiu}) we deduce that $T_k=q_k^*S_k$ and that $E_k=q_k^*D_k$. In particular, letting $\gamma_k:=\{S_k\}$ and $\beta_k:=\{T_k\}$, we have $\beta_k=q_k^*\gamma_k$. Observe that $\beta_k$ are the nef classes on $Y_k$ associated to $\varphi_k$ through Lemma \ref{lem:ClassicFujita}. By Projection Formula and \ref{eqn:New_Oct} we deduce that
    $$
    \limsup_{k\to +\infty} \beta_k^{n-1}\cdot K_{Y_k/X}= \limsup_{k\to +\infty} \gamma_k^{n-1}\cdot K_{Z_k/X}\leq 2nk_0 \limsup_{k\to +\infty} \gamma_k^{n-1}\cdot \pi_k^*E=2nk_0\langle \alpha^{n-1}\rangle\cdot E=0, 
    $$
    where the second-to-last equality follows, similarly to the proof of Lemma \ref{lem:Special_H}, from the fact that $\left(\pi_k:Z_k\to X, \gamma_k,\frac{D_k}{k+k_0}\right)$ is associated with the Nef Fujita Approximation $\left(p_k:Y_k\to X,\beta_k,\frac{E_k}{k+k_0}\right)$, while the last equality is given by Theorem \ref{conj:Ort} since $\mathrm{Supp}(E)=\mathrm{E}_{\mathrm{nK}}(\alpha)$. Hence, by Lemma \ref{lem:Special_H}, the chosen Nef Fujita Approximation is Special and this concludes the proof.
    % $$
    % \beta_k^{n-1}\cdot q_k^*K_{Z_k/X}\leq n k_0\, \beta_k^{n-1}\cdot q_k^*\pi_k^*E = nk_0\, \beta_k^{n-1}\cdot p_k^*E,
    % $$
    % where we denoted by $\beta_k$ the nef classes on $Y_k$ associated to $\varphi_k$ through Lemma \ref{lem:ClassicFujita}. Hence
    % \begin{equation}
    %     \label{eqn:NewHome1}
    %     \limsup_{k\to +\infty} \beta_k^{n-1}\cdot q_k^* K_{Z_k/Y_k}\leq nk_0\, \limsup_{k\to +\infty} \beta_k^{n-1}\cdot p_k^*E= nk_0\, \langle \alpha^{n-1} \rangle\cdot E=0,
    % \end{equation}
    % where the second-to-last equality can be easily checked as in the proof of Lemma \ref{lem:Special_H}, while the last equality follows from Theorem \ref{conj:Ort} since $\mathrm{Supp}(E)=\mathrm{E}_{\mathrm{nK}}(\alpha)$.

    % By Projection Formula we deduce that
    % \begin{equation}
    %     \label{eqn:NewHome2}
    %     \beta_k^{n-1}\cdot K_{Y_k/Z_k}=q_k^*\gamma_k^{n-1}\cdot K_{Y_k/Z_k}=0.
    % \end{equation}
    % Since $K_{Y_k/X}=q_k^*K_{Z_k/X}+K_{Y_k/Z_k}$, combining \eqref{eqn:NewHome1} and \eqref{eqn:NewHome2} we obtain
    % $$
    % \limsup_{k\to +\infty}\beta_k^{n-1}\cdot K_{Y_k/X}=0,
    % $$
    % i.e. the chosen Nef Fujita Approximation is Special by Lemma \ref{lem:Special_H}. This concludes the proof.
    \end{proof}

    \begin{lem}\label{lem:Uniformity}
        Let $u$ be a q-psh function such that $c_1\theta_1+\cdots +c_M\theta_M+\ddc u\geq -c \om$ for $\theta_j$ smooth closed $(1,1)$-forms, $\om$ K\"ahler form and $c_j>0,c>0$. Then there exists a q-psh function $\tilde{u}$ with analytic singularities of type $\left(\mathcal{J}(u),1\right)$ such that $c_1\theta_1+\cdots+ c_M\theta_M+\ddc \tilde{u}\geq -(c+A)\om$ for a constant $A>0$ only depending on $\om$ and on $\theta_j$.
    \end{lem}
    \begin{proof}
        This result follows from \cite{Pop03}, which in turn is based on the gluing procedure of \cite{Dem92}. For completeness, we sketch the proof.
        
        Using the same notation as in \cite{Pop03}, let $ \mathcal{U}:=\{\B_j\}_j, \mathcal{U}^{''}:=\{\B^{''}_j\}_j, \mathcal{U}^{'}:=\{\B_j^{'}\}_j, \mathcal{U}^{(3)}:=\{\B_j^{(3)}\}$ be finite coverings of $X$ given by concentric balls of radii respectively $2\delta, \frac{3}{2}\delta, \delta$ and $\delta/2$ for $\delta>0$, centered in holomorphic coordinates $z^j:=(z_1^j,\dots,z_n^j)$. Without loss of generality we can suppose that $\theta_s=\ddc f_j^s, \om=\ddc g_j$ over $\B_j$ for smooth functions $f_j^s, g_j$, for all $j$ and $s=1,\dots, M$. Since $f_j^s-f_k^s, g_j-g_k$ are pluriharmonic, after possibly shrinking the balls $B_j$, we can assume that $f_j^s-f_k^s=\mathrm{Re}\, f^s_{jk}, g_j-g_k=\mathrm{Re}\, g_{jk}$ on $\B_j\cap \B_k$ for holomorphic functions $f_{jk}^s, g_{jk}$ on $\B_j\cap \B_k$. Note that this setting only depends on $\om, \theta_s$ but not on $c_s, c$. We also set $h_j:=-\sum_{s=1}^M c_s f_j^s-c g_j, h_{jk}:=-\sum_{s=1}^M c_s f^s_{jk}-c g_{jk}$, so that $h_j-h_k=\mathrm{Re}\, h_{jk}$ on $\B_j\cap \B_k$. 
        
        Next, the functions $v_j:= u - h_j$ are plurisubharmonic on $\B_j$ for any $j$. We then let $\{\sigma_{j,l}\}_l$ be an orthonormal basis of the Hilbert space $\mathcal{H}_{\B_j}(v_j):=\left\{f\in \mathcal{O}(\B_j)\, : \, \int_{\B_j}\lvert f_j\rvert^2e^{-v_j}\,d\lambda<+\infty \right\}$, and we set $\psi_j:=\log \sum_{l}\lvert \sigma_{j,l} \rvert^2 $ and $\varphi_j:=\psi_j+h_j$. The key point is to glue these quasi-psh functions. Indeed, it is shown in \cite{Pop03} how the function
        $$
        \tilde{u}(z):=\sup_{z\in \B^{''}_j} \left\{\varphi_j(z) + C_1\left(\delta^2-\lVert z^j\rVert^2\right) \right\}
        $$
        gives a q-psh function such that $\ddc \tilde{u}\geq -\sum_s c_s\theta_s-c\om- C_1 C_2\om$ where $C_2$ is such that $\ddc \lVert z^j \rVert^2\leq C_2\om$ on $\B_j$ for all $j$ while $C_1$ is chosen sufficiently large so that $\varphi_j(z)+C_1\left(\delta^2-\lVert z^j\rVert^2\right)\leq \varphi_k(z)+C_1\left(\delta^2-\lVert z^k\rVert^2\right) $ for any $z\in \B_k^{(3)}\cap\left(\overline{\B''_j}\setminus \B'_j\right)$, for all $k, j$. In particular, $\tilde{u}$ has analytic singularities of type $\left(\mathcal{J}(u),1\right)$ (cf. \cite[Prop. 1.2]{DEL00}) and to conclude the proof it remains to check the dependence of $C_1$. But as explained at the end of p. $63$ of \cite{Pop03}, by \cite[Lem. 1.2]{Pop03}, the constant $C_1$ can be chosen to depend only on $\delta$.
    \end{proof}

    \section{Solution to the transcendental Yau-Tian-Donaldson Conjecture}\label{sec:6}

    In this section we conclude the proof of the Yau-Tian-Donaldson Conjecture of Theorem \ref{thmA}.

    \begin{proof}[Proof of Theorem \ref{thmA}]
        The implication $(i)\Rightarrow (ii)$ has already been proved in \cite[Cor. 5.5]{Sjo18}, \cite[Thm. 1.1]{DR17}. Thus, thanks to Theorem \ref{thm:Pietro_David}, it is enough to prove that uniform $K$-stability implies the analogous notion for models.

        Let $(\cX,\cA+\cD)$ be a fixed big test configuration (we recall that by our assumptions, this in particular means that $(\cX,\cA+\cD)$ is dominating and smooth, cf. subsection \ref{ssec:K-stability}). An easy check shows that the Donaldson-Futaki invariant and the Non-Archimedean $J$-functional are both invariant under the change $\cD\leadsto \cD+c\cX_0$ for $c\geq 0$. Thus we can assume that $\cD\geq \cX_0$ without loss of generality. It follows that $\mathrm{E}_{\mathrm{nK}}(\cA+\cD)\subset \cX_0$ as we can produce a K\"ahler current with analytic singularities that is smooth outside the zero fiber.

        By Theorem \ref{thmB} there exists a Special K\"ahler Fujita Approximation of $(\cX,\cA+\cD)$ given by blowing-up smooth centers in the zero fibers. Moreover by construction of the Special K\"ahler Fujita Approximation in Theorem \ref{thmB} such smooth centers can be chosen to be $\C^*$-invariant. In particular, the $\C^*$-action lifts through each blow-up, and hence the resulting approximations are Kähler test configurations $(\cX_k,\cA_k+\cD_k)$ such that
        \begin{equation}
            \label{eqn:Needed}\langle(\cA_k+\cD_k)^{n+1}\rangle\longrightarrow\langle (\cA+\cD)^{n+1} \rangle,\quad \langle(\cA_k+\cD_k)^n\rangle\cdot K_{\cX_k}\longrightarrow \langle (\cA+\cD)^n\rangle\cdot K_{\cX}
        \end{equation}
        as $k\to +\infty$. Moreover, denoting by $p_k:\cX_k\to \cX$ and by $\pi:\cX\to X\times \P^1$ the modifications obtained by this construction, from the fact that $\{(\cX_k,\cA_k+\cD_k)\}_k$ is associated to a K\"ahler Fujita Approximation of $(\cX,\cA+\cD)$ we also have that $\langle(\cA_k+\cD_k)^n\rangle\cdot p_k^*\pi^*p_2^*K_{\P^1}\to \langle(\cA+\cD)^n\rangle\cdot\pi^*p_2^*K_{\P^1} $ as $k\to +\infty$ (cf. the proof of Lemma \ref{lem:Special_H}). This together with \eqref{eqn:Needed} leads to
        \begin{equation}
            \label{eqn:1Final}
            \DF(\cX_k,\cA_k+\cD_k)\to \DF(\cX,\cA+\cD)
        \end{equation}
        as $k\to +\infty$.
        Furthermore, in the Non-Archimedean formalism developed in \cite{BJ23, MP25, MPWN25}, any K\"ahler Fujita Approximation $\{(\cX_k,\cA_k+\cD_k)\}_k$ of $(\cX,\cA+\cD)$ produces a sequence of Non-Archimedean metrics converging \emph{strongly} to the Non-Archimedean metric associated with $(\cX,\cA+\cD)$. Hence
        \begin{equation}
            \label{eqn:2Final}
            \jj^\NA(\cX_k,\cA_k+\cD_k)\longrightarrow \jj^\NA(\cX,\cA+\cD)
        \end{equation}
        as $k\to +\infty$. Indeed, this convergence follows from $\cA_k^n\cdot \langle \cA_k+\cD_k\rangle\to \cA^n\cdot \langle \cA+\cD\rangle$, which can be for instance deduced from \cite[Lem. 5.0.5]{MPWN25}. Combining \eqref{eqn:1Final} with \eqref{eqn:2Final} then shows that uniform $K$-stability implies uniform $K$-stability for models, concluding the proof.
    \end{proof}

    \begin{rmk}\label{rmk:Final}
        A close examination of the proof of Theorem \ref{thmB} shows that we invoke the recent breakthrough Theorem \ref{conj:Ort} only because of \eqref{eqn:Orthogonality}.
        %Although the recent breakthrough of Theorem \ref{conj:Ort} is used to get Theorem
        In particular, Theorem \ref{conj:Ort} is not necessary to obtain Theorem \ref{thmA}. Indeed, for any big test configuration $(\cX,\cA+\cD)$ and for any $\cE$ prime divisor contained in $\mathrm{E}_{\mathrm{nK}}(\cA+\cD)$, the equality
        \begin{equation*}%\label{eqn:NeededB}
        \langle (\mathcal{A}+\mathcal{D})^n\rangle \cdot \cE=0
        \end{equation*}
        follows from combining \cite[Prop. 9.1.2]{MPWN25} with \cite[Thm. 1.1]{Vu23} (cf. also \cite{Nys24}).
        %is not actually necessary for the proof of Theorem \ref{thmA}. Indeed a close examination of the proof of Theorem \ref{thmB} shows that any big test configuration of $(X,\alpha)$ admits a Special K\"ahler Fujita Approximation given by K\"ahler test configurations provided every big test configuration $(\cX,\cA+\cD)$ satisfies
        %\begin{equation}\label{eqn:NeededB}
        %\langle (\mathcal{A}+\mathcal{D})^n\rangle \cdot \cE=0
        %\end{equation}
        %for any $\cE$ prime divisor contained in $\mathrm{E}_{\mathrm{nK}}(\cA+\cD)$. The equality \eqref{eqn:NeededB} follows from combining \cite[Prop. 9.1.2]{MPWN25} with \cite[Thm. 1.1]{Vu23} (cf. also \cite{Nys24}).
    \end{rmk}

{\footnotesize
\bibliographystyle{acm}
\bibliography{main}
}
\end{document}